\documentclass[11pt, twoside, leqno]{article}

\usepackage{amssymb}
\usepackage{amsmath}
\usepackage{amsthm}
\usepackage{color}
\usepackage{mathrsfs}
\usepackage{txfonts}

\usepackage{indentfirst}

\allowdisplaybreaks

\def\rr{{\mathbb R}}
\def\rn{{\mathbb{R}^n}}
\def\zz{{\mathbb Z}}

\def\nn{{\mathbb N}}
\def\cp{{\mathcal P}}
\def\cs{{\mathcal S}}

\def\fz{\infty }
\def\az{\alpha}

\def\gz{\gamma}
\def\lz{\lambda}

\def\lf{\left}
\def\r{\right}
\def\hs{\hspace{0.25cm}}
\def\ls{\lesssim}
\def\noz{\nonumber}

\def\supp{\mathop\mathrm{\,supp\,}}
\def\esup{\mathop\mathrm{\,ess\,sup\,}}
\def\B{\mathfrak{B}}

\def\vh{{H_A^{p(\cdot)}(\rn)}}

\def\vah{{H_A^{p(\cdot),r,s}(\rn)}}

\def\lfz{{L^{\fz}(\rn)}}
\def\lv{{L^{p(\cdot)}(\rn)}}

\newtheorem{theorem}{Theorem}[section]
\newtheorem{lemma}[theorem]{Lemma}

\theoremstyle{definition}
\newtheorem{remark}[theorem]{Remark}
\newtheorem{definition}[theorem]{Definition}
\renewcommand{\appendix}{\par
   \setcounter{section}{0}%
   \setcounter{subsection}{0}%
   \setcounter{subsubsection}{0}%
   \gdef\thesection{\@Alph\c@section}%
   \gdef\thesubsection{\@Alph\c@section.\@arabic\c@subsection}%
   \gdef\theHsection{\@Alph\c@section.}%
   \gdef\theHsubsection{\@Alph\c@section.\@arabic\c@subsection}%
   \csname appendixmore\endcsname
 }

\numberwithin{equation}{section}

\begin{document}

\arraycolsep=1pt

\title{\bf\Large Fourier Transform of Variable Anisotropic Hardy Spaces
with Applications to Hardy--Littlewood Inequalities
\footnotetext{\hspace{-0.35cm} 2010 {\it
Mathematics Subject Classification}. Primary 42B35;
Secondary 42B30, 42B10, 46E30.
\endgraf {\it Key words and phrases.}
expansive matrix, (variable) Hardy space, Fourier transform,
Hardy--Littlewood inequality.
\endgraf The author is supported by the Fundamental Research Funds for the Central Universities
(Grant No.~2020QN21), the Natural Science Foundation of Jiangsu Province
(Grant No.~BK20200647)
and the National Natural Science Foundation of
China (Grant No.~12001527).}}
\author{Jun Liu}
\date{}
\maketitle

\vspace{-0.8cm}

\begin{center}
\begin{minipage}{13cm}
{\small {\bf Abstract}\quad
Let $p(\cdot):\ \mathbb{R}^n\to(0,1]$ be a variable exponent
function satisfying the globally log-H\"{o}lder continuous condition
and $A$ a general expansive matrix on $\mathbb{R}^n$.
Let $H_A^{p(\cdot)}(\mathbb{R}^n)$ be the variable anisotropic
Hardy space associated with $A$ defined via the radial maximal function.
In this article, via the known atomic characterization of
$H_{A}^{p(\cdot)}(\mathbb{R}^n)$ and establishing two useful estimates
on anisotropic variable atoms, the author shows that the Fourier transform
$\widehat{f}$ of $f\in H_A^{p(\cdot)}(\mathbb{R}^n)$ coincides with a
continuous function $F$ in the sense of tempered distributions,
and $F$ satisfies a pointwise inequality which contains a step function
with respect to $A$
as well as the Hardy space norm of $f$. As applications, the author
also obtains a higher order convergence of the continuous function $F$ at
the origin. Finally, an analogue of the
Hardy--Littlewood inequality in the variable anisotropic Hardy space setting
is also presented. All these
results are new even in the classical isotropic setting.}
\end{minipage}
\end{center}

\vspace{0.2cm}

\section{Introduction\label{s1}}

The main purpose of this article is to investigate the Fourier transform
on the variable anisotropic Hardy space $\vh$ from \cite{lwyy18},
where $p(\cdot):\ \rn\to(0,1]$
is a variable exponent function satisfying the so-called globally log-H\"{o}lder continuous condition [see \eqref{2e4} and \eqref{2e5} below] and $A$ a general
expansive matrix on $\rn$ (see Definition \ref{2d1} below).
As we all know, the problem of the Fourier transform on classical Hardy spaces $H^p({\mathbb{R}^n})$ originates from Fefferman and Stein \cite{fs72},
which is a very interesting and hot topic in the real-variable theory of
$H^p({\mathbb{R}^n})$. First, using entire functions of exponential type,
Coifman \cite{c74} characterized the Fourier transform $\widehat{f}$ of
$f\in H^p({\mathbb R})$ (namely, for the dimension $n=1$).
For the study of the Fourier transform on Hardy spaces in the higher dimensions,
we refer the reader to \cite{bw13,c82,gk01,tw80} and their references.

In particular, the following well-known result was obtained by
Taibleson and Weiss \cite{tw80}: for each fixed $p\in(0,1]$,
the Fourier transform $\widehat{f}$ of
$f$ which belongs to $H^p({\mathbb{R}^n})$ coincides with a continuous
function $F$ in the sense of tempered distributions and, for each
$\xi\in{\mathbb{R}^n}$,
\begin{align}\label{1e1}
\left|F(\xi)\right|\le C\|f\|_{H^p({\mathbb{R}^n})}|\xi|^{n(1/p-1)},
\end{align}
where $C$ is a positive constant depending only on $n$ and $p$. Moreover,
the inequality \eqref{1e1} further implies the famous Hardy--Littlewood
inequality for Hardy spaces, namely, for each given $p\in(0,1]$, there
exists a positive constant $R$ such that, for any $f\in H^p({\mathbb{R}^n})$,
\begin{align}\label{1e2}
\left[\int_{{\mathbb{R}^n}}|\xi|^{n(p-2)}
\left|F(\xi)\right|^p\,d\xi\right]^{1/p}\le R\|f\|_{H^p({\mathbb{R}^n})},
\end{align}
where $F$ is as in \eqref{1e1}; see, for instance, \cite[p.\,128]{s93}.
In addition, via the known atomic characterization of the anisotropic
Hardy space $H^p_A({\mathbb{R}^n})$,
Bownik and Wang \cite{bw13} proved that both inequalities \eqref{1e1} and
\eqref{1e2} hold true for the Hardy space $H^p_A({\mathbb{R}^n})$.
Very recently, these results were extended to the setting of
Hardy spaces associated with ball
quasi-Banach function spaces or
the anisotropic mixed-norm Hardy space ${H_{\vec{a}}^{\vec{p}}(\rn)}$, where
$$\vec{a}:=(a_1,\ldots,a_n)\in [1,\infty)^n\quad\mathrm{and}\quad\vec{p}:=(p_1,\ldots,p_n)\in (0,1]^n$$
are two vectors; see \cite{hcy,hcy21}.

On another hand, as a generalization of the classical Hardy space $H^p(\rn)$,
the variable Hardy space $H^{p(\cdot)}(\rn)$, in which the constant exponent
$p$ is replaced by a variable
exponent function $p(\cdot):\ \rn\to(0,\fz]$, was first studied by Nakai and Sawano
\cite{ns12} and, independently, by Cruz-Uribe and Wang \cite{cw14} with some weaker assumptions on $p(\cdot)$ than those used in \cite{ns12}.
For more development about this Hardy space and other function spaces with variable exponents, we refer the reader to
\cite{ah10,abr16,cf13,dhhr11,dhr09,jzz17,kv14,s10,xu08,yyyz16,yzn16,zsy16,zyy18}.
In addition, the anisotropic Hardy space $H_A^p(\rn)$, with $p\in(0,\fz)$,
was first investigated by Bownik \cite{mb03}, which have proved important for the
study of discrete groups of dilations in wavelet theory,
and also includes both the classical
Hardy space and the parabolic Hardy space of Calder\'{o}n and Torchinsky \cite{ct75}
as special cases. Based on these work,
recently, Liu et al. \cite{lwyy18} introduced the variable anisotropic Hardy space
$\vh$ with respect to the expansive matrix $A$, and established
its various real-variable characterizations.
Nowadays, this anisotropic setting has proved useful not only
in developing function spaces arising in harmonic analysis,
but also in many other branches such as the wavelet theory
(see, for instance, \cite{bb12,mb03,dhp09}) and partial differential equations
(see, for instance, \cite{bw19,jm06}).

Motivated by the real-variable theory of the variable anisotropic Hardy space
$\vh$ from \cite{lwyy18}
and the aforementioned results about the characterizations
of the Fourier transform on
classical Hardy spaces $H^p({\mathbb{R}^n})$ and anisotropic Hardy spaces $H^p_A({\mathbb{R}^n})$ as well as anisotropic mixed-norm Hardy spaces ${H_{\vec{a}}^{\vec{p}}(\rn)}$, in this article, we first extend the inequality
\eqref{1e1} to the setting of variable anisotropic Hardy spaces and
then also give out some applications of our main result.

To be precise, in Section \ref{s2}, we recall the notions of expansive matrices,
variable Lebesgue spaces $\lv$ and variable anisotropic Hardy spaces
(see, Definitions \ref{2d1} and \ref{2d4} below).

The target of Section \ref{s3} is to obtain the main result,
namely, Theorem \ref{3t1} below. For this purpose, we first
establish two uniform pointwise estimates on anisotropic variable atoms
(see Lemmas \ref{3l1} and \ref{3l2} below) as well as an auxiliary inequality
(see Lemma \ref{3l5} below).
Using these and the known atomic characterization of $\vh$ from
\cite[Theorem 4.8]{lwyy18}, we then
show that the Fourier transform $\widehat{f}$
of $f$ which belongs to $\vh$ coincides with a continuous
function $F$ in the sense of tempered distributions. We also prove
that this continuous function $F$, multiplied by a step function with
respect to $A$, can be pointwisely controlled by a positive constant multiple
of the Hardy space norm of $f$.
This elucidates the necessity of vanishing
moments of anisotropic variable atoms in some sense [see Remark \ref{3r1}(ii)
below].

In Section \ref{s4}, as applications, applying a technical
inequality obtained in the proof of Theorem \ref{3t1},
we first show a higher order convergence
of the continuous function $F$ at the origin (see Theorem \ref{4t1} below).
Then we prove that the function $F$,
multiplied by some power of a step function with
respect to $A$, is $p_+$-integrable,
and this integral can be controlled by a positive constant multiple
of the Hardy space norm of $f$ (see Theorem \ref{4t2} below).
This result is a generalization of the Hardy--Littlewood inequality
for the present setting of variable anisotropic Hardy spaces.

Finally, we make some conventions on notation.
We always let $\nn:=\{1,2,\ldots\}$, $\zz_+:=\{0\}\cup\nn$
and $\mathbf{0}$ be the \emph{origin} of $\rn$. For each fixed multi-index
$\az:=(\az_1,\ldots,\az_n)\in(\zz_+)^n=:\zz_+^n$,
let $|\az|:=\az_1+\cdots+\az_n$ and
$\partial^{\az}
:=(\frac{\partial}{\partial x_1})^{\az_1}\cdots(\frac{\partial}{\partial x_n})^{\az_n}.$
We use $C$ to denote a positive constant
which is independent of the main parameters, but may vary in different setting.
The \emph{symbol} $g\ls h$ means $g\le Ch$ and,
if $g\ls h\ls g$, then we write $g\sim h$. If $f\le Ch$ and $h=g$ or $h\le g$,
we then write $f\ls h\sim g$ or $f\ls h\ls g$, \emph{rather than} $f\ls h=g$
or $f\ls h\le g$. In addition,
for any set $E\subset\rn$, we denote by $\mathbf{1}_E$ its \emph{characteristic function},
by $E^\complement$ the
set $\rn\setminus E$ and by $|E|$ its \emph{n-dimensional Lebesgue measure}.
For any $d\in\mathbb{R}$, we denote by $\lfloor d\rfloor$
the \emph{largest integer not greater than $d$}.

\section{Preliminaries \label{s2}}

In this section, we recall the notions of expansive matrices
and variable anisotropic Hardy spaces (see, for instance, \cite{mb03,lwyy18}).

The following definition of expansive matrices is from \cite{mb03}.

\begin{definition}\label{2d1}
A real $n\times n$ matrix $A$ is called an \emph{expansive matrix}
(shortly, a \emph{dilation}) if
$$\min_{\lz\in\sigma(A)}|\lz|>1,$$
here and thereafter, $\sigma(A)$ denotes the \emph{set of all eigenvalues of $A$}.
\end{definition}

By \cite[p.\,5, Lemma 2.2]{mb03}, we know that there exist an open
ellipsoid $\Delta$, with $|\Delta|=1$, and $r\in(1,\infty)$ such that
$\Delta\subset r\Delta\subset A\Delta$. Thus,
for any $i\in\zz$, $B_i:=A^i\Delta$ is open,
$B_i\subset rB_i\subset B_{i+1}$ and $|B_i|=b^i$ with $b:=|\det A|$.
For any $x\in\rn$ and $i\in\zz$, an ellipsoid $x+B_i$
is called a \emph{dilated ball}. Let $\mathfrak{B}$ be the set of all such
dilated balls, namely,
\begin{align}\label{2e1}
\mathfrak{B}:=\lf\{x+B_i:\ x\in\rn,\ i\in\zz\r\}
\end{align}
and let $\tau:=\inf\{k\in\zz:\ r^k\ge2\}.$

In \cite[p.\,6, Definition 2.3]{mb03},
the following homogeneous quasi-norm was introduced.

\begin{definition}\label{2d2}
Given a dilation $A$, a measurable mapping $\rho:\ \rn \to [0,\infty)$
is called a \emph{homogeneous quasi-norm}, with respect to $A$, if
\begin{enumerate}
\item[\rm{(i)}] $x\neq\mathbf{0}$ implies $\rho(x)\in(0,\fz)$;

\item[\rm{(ii)}] for any $x\in\rn$, $\rho(Ax)=b\rho(x)$;

\item[\rm{(iii)}] there exists a constant $C\in[1,\fz)$
such that, for any $x$, $y\in\rn$, $\rho(x+y)\le C[\rho(x)+\rho(y)]$.
\end{enumerate}
\end{definition}

For any given dilation $A$, in \cite[p.\,6, Lemma 2.4]{mb03},
it was proved that all homogeneous
quasi-norms with respect to $A$ are equivalent.
Therefore, once $A$ is fixed, we can use the
\emph{step homogeneous quasi-norm} $\rho$ defined by setting, for any $x\in\rn$,
\begin{equation*}
\rho(x):=\left\{
\begin{array}{cl}
b^i&\hspace{0.6cm} {\rm when}\hspace{0.5cm} x\in
B_{i+1}\backslash B_i,\\
0&\hspace{0.6cm} {\rm when}\hspace{0.5cm} x=\mathbf{0}
\end{array}\r.
\end{equation*}
for convenience.

Recall also that an infinitely differentiable function $\phi$ is called a
\emph{Schwartz function} if, for any $k\in\zz_+$ and multi-index $\gz\in\zz_+^n$,
$$\|\phi\|_{\gz,k}:=
\sup_{x\in\rn}[\rho(x)]^k\lf|\partial^\gz\phi(x)\r|<\infty.$$
Let $\cs(\rn)$ be the set of all Schwartz functions as above, equipped with the topology
determined by $\{\|\cdot\|_{\alpha,\ell}\}_{\az\in\zz_+^n,\,\ell\in\zz_+}$, and
$\cs'(\rn)$ its \emph{dual space}, equipped with the weak-$\ast$
topology.
Throughout this article, for any $\phi\in\cs(\rn)$ and $i\in\zz$, let
$\phi_i(\cdot):=b^{i}\phi(A^{i}\cdot)$.

For any measurable function $p(\cdot):\ \rn\to(0,\fz]$, let
\begin{align}\label{2e3}
p_-:=\mathop\mathrm{ess\,inf}_{x\in \rn}p(x),\hspace{0.35cm}
p_+:=\mathop\mathrm{ess\,sup}_{x\in \rn}p(x)\hspace{0.35cm}
{\rm and}\hspace{0.35cm}\underline{p}:=\min\{p_-,1\}.
\end{align}
Denote by $\cp(\rn)$ the \emph{set of all measurable functions}
$p(\cdot)$ satisfying $0<p_-\le p_+<\fz$.

Given a function $p(\cdot)\in\cp(\rn)$,
the \emph{modular functional} $\varrho_{p(\cdot)}$ and
the \emph{Luxemburg--Nakano quasi-norm} $\|f\|_{\lv}$,
with respect to $p(\cdot)$, are defined, respectively, by setting,
for any measurable function $f$,
$$\varrho_{p(\cdot)}(f):=\int_\rn|f(x)|^{p(x)}\,dx
\quad{\rm and}\quad
\|f\|_{\lv}:=\inf\lf\{\lz\in(0,\fz):\ \varrho_{p(\cdot)}(f/\lz)\le1\r\}.$$
Furthermore, the \emph{variable Lebesgue space} $\lv$ is defined to be the
set of all measurable functions $f$ such that $\varrho_{p(\cdot)}(f)<\fz$,
equipped with the quasi-norm $\|f\|_{\lv}$.

Let $C^{\log}(\rn)$ be the set of all $p(\cdot)\in\cp(\rn)$ satisfying the
\emph{globally log-H\"older continuous condition}, which means
there exist two positive constants $C_{\log}(p)$ and $C_\fz$, and
$p_\fz\in\rr$ such that, for any $x,\ y\in\rn$,
\begin{equation}\label{2e4}
|p(x)-p(y)|\le \frac{C_{\log}(p)}{\log(e+1/\rho(x-y))}
\end{equation}
and
\begin{equation}\label{2e5}
|p(x)-p_\fz|\le \frac{C_\fz}{\log(e+\rho(x))}.
\end{equation}

\begin{definition}\label{2d3}
Let $\phi\in\cs(\rn)$ satisfy $\int_\rn\phi(x)\,dx\neq0$. The
\emph{radial maximal function} $M_\phi(f)$ of $f\in\cs'(\rn)$,
with respect to $\phi$, is defined by setting, for any $x\in\rn$,
\begin{align*}
M_\phi(f)(x):= \sup_{i\in\zz}|f\ast\phi_i(x)|.
\end{align*}
\end{definition}

Applying \cite[Definition 2.4 and Theorem 3.10]{lwyy18}, we now give
a equivalent definition of variable anisotropic Hardy spaces
as follows.

\begin{definition}\label{2d4}
Let $p(\cdot)\in C^{\log}(\rn)$ and $\phi$ be as in Definition \ref{2d3}.
The \emph{variable anisotropic Hardy space} $\vh$ is defined by setting
\begin{equation*}
\vh:=\lf\{f\in\cs'(\rn):\ M_\phi(f)\in\lv\r\}
\end{equation*}
and, for any $f\in\vh$, let
$\|f\|_{\vh}:=\| M_\phi(f)\|_{\lv}$.
\end{definition}

\section{Fourier transforms of $\vh$\label{s3}}

In this section, we study the Fourier transform $\widehat{f}$,
where the distribution $f$ comes from the variable anisotropic Hardy space $\vh$.

Recall that, for any $\phi\in\cs(\rn)$, its \emph{Fourier transform},
denoted by $\mathfrak{F}\phi$ or $\widehat{\phi}$, is defined by setting,
for any $v\in\rn$,
$$\mathfrak{F}\phi(v)=\widehat\phi(v)
:=\int_{\rn}\phi(x)e^{-2\pi\imath x\cdot v}\,dx.
$$
here and thereafter, $\imath:=\sqrt{-1}$ and $x\cdot v := \sum_{k=1}^n x_k v_k$
for any $x:=(x_1,\ldots,x_n)$, $v:=(v_1,\ldots,v_n)\in\rn$.
Moreover, for any $f\in\cs'(\rn)$, its \emph{Fourier transform},
also denoted by $\mathfrak{F}f$ or $\widehat{f}$, is defined by setting,
for any $\phi\in\cs(\rn)$,
$$\langle\mathfrak{F}f,\phi\rangle=\langle\widehat f,\phi\rangle
:=\langle f,\widehat{\phi}\rangle.$$

We now present the main result of this article as follows:
the Fourier transform $\widehat{f}$ of $f\in\vh$ coincides with a
continuous function $F$ in the sense of tempered distributions,
and $F$ satisfies a pointwise inequality.

\begin{theorem}\label{3t1}
Let $p(\cdot)\in C^{\log}(\rn)$ satisfy $0<p_-\le p_+\le1$,
where $p_-,\ p_+$ are as in \eqref{2e3}. Then, for any
$f\in\vh$, there exists a continuous function $F$ on $\rn$ such that
$$\widehat{f}=F\quad in\quad \cs'(\rn),$$
and there exists a positive constant $C$, depending only on
$A$, $p_-$ and $p_+$, such that, for any $x\in\rn$,
\begin{align}\label{3e6}
|F(x)|\le C\|f\|_{\vh}\max\lf\{[\rho_\ast(x)]^{\frac1{p_-}-1},
[\rho_\ast(x)]^{\frac1{p_+}-1}\r\},
\end{align}
here and thereafter, $\rho_\ast$ is as in Definition \ref{2d2} with $A$ replaced by its
adjoint matrix $A^*$.
\end{theorem}

To prove this theorem, we need some notions and technical lemmas.
First, for any $r\in(0,\fz]$ and measurable set $E\subset\rn$, the Lebesgue
space $L^r(E)$ is defined to be the set of all measurable functions $f$ such that,
when $r\in(0,\fz)$,
$$\|f\|_{L^r(E)}:=\lf[\int_E|f(x)|^r\,dx\r]^{1/r}<\fz$$
and
$$\|f\|_{L^\fz(E)}:=\esup_{x\in E}|f(x)|<\fz.$$

In addition, the dilation operator $D_A$ is defined by setting, for any measurable
function $f$ on $\rn$,
$$D_A(f)(\cdot):=f(A\cdot).$$
Moreover, we have the following identity: for any $k\in\zz$, $f\in L^1(\rn)$
and $x\in\rn$,
$$\widehat{f}(x)=b^k\lf(D_{A^*}^k\mathfrak{F}D_{A}^kf\r)(x).$$

The succeeding notions of both anisotropic $(p(\cdot),r,s)$-atoms and
variable anisotropic atomic Hardy spaces $\vah$ are from \cite{lwyy18}.

\begin{definition}\label{3d3}
\begin{enumerate}
\item[{\rm (i)}]
Let $p(\cdot)\in\cp(\rn)$, $r\in(1,\fz]$,
\begin{align}\label{3e2}
s\in\lf[\lf\lfloor\lf(\dfrac1{p_-}-1\right)
\dfrac{\ln b}{\ln\lz_-}\right\rfloor,\fz\right)\cap\zz_+,
\end{align}
where $p_-$ is as in \eqref{2e3}.
A measurable function $a$ on $\rn$ is called an \emph{anisotropic $(p(\cdot),r,s)$-atom}
(shortly, a $(p(\cdot),r,s)$-\emph{atom}) if
\begin{enumerate}
\item[{\rm (i)$_1$}] $\supp a \subset B$, where
$B\in\B$ with $\B$ as in \eqref{2e1};

\item[{\rm (i)$_2$}] $\|a\|_{L^r(\rn)}\le \frac{|B|^{1/r}}{\|\mathbf{1}_B\|_{\lv}}$;

\item[{\rm (i)$_3$}] $\int_{\rn}a(x)x^\gz\,dx=0$ for any $\gz\in\zz_+^n$
with $|\gz|\le s$.
\end{enumerate}
\item[{\rm (ii)}]
Let $p(\cdot)\in C^{\log}(\rn)$, $r\in(1,\fz]$ and $s$ be as in \eqref{3e2}.
The \emph{variable anisotropic atomic Hardy space} $\vah$ is defined to be the
set of all $f\in\cs'(\rn)$ satisfying that there exist a sequence
$\{\lz_i\}_{i\in\nn}\subset\mathbb{C}$ and a sequence of $(p(\cdot),r,s)$-atoms,
$\{a_i\}_{i\in\nn}$, supported, respectively, in
$\{B^{(i)}\}_{i\in\nn}\subset\B$ such that
\begin{align*}
f=\sum_{i\in\nn}\lz_ia_i
\quad\mathrm{in}\quad\cs'(\rn).
\end{align*}
Moreover, for any $f\in\vah$, let
\begin{align*}
\|f\|_{\vah}:=
{\inf}\lf\|\lf\{\sum_{i\in\nn}
\lf[\frac{|\lz_i|\mathbf{1}_{B^{(i)}}}{\|\mathbf{1}_{B^{(i)}}\|_{\lv}}\r]^
{\underline{p}}\r\}^{1/\underline{p}}\r\|_{\lv},
\end{align*}
where the infimum is taken over all the decompositions of $f$ as above.
\end{enumerate}
\end{definition}

Applying the vanishing moment of atoms and the Taylor remainder theorem,
we obtain the following uniform estimate for atoms.

\begin{lemma}\label{3l1}
Let $p(\cdot)$, $r$ and $s$ be as in Definition \ref{3d3}(ii).
Assume that $a$ is a $(p(\cdot),r,s)$-atom supported in $x_0+B_{k_0}$
with some $x_0\in\rn$ and $k_0\in\zz$. Then there exists a positive
constant $C$, depending only on $A$ and $s$, such that,
for any $\alpha\in\zz_+^n$ with $|\az|\leq s$ and $x\in\rn$,
\begin{align}\label{3e1}
\lf|\partial^\az\lf(\mathfrak{F}D_{A}^{k_0}a\r)(x)\r|
\leq Cb^{-k_0/r}\|a\|_{L^r(\rn)}\min\lf\{1,|x|^{s-|\az|+1}\r\}.
\end{align}
\end{lemma}

\begin{proof}
Without loss of generality, we may assume that $a$ is supported in $B_{k_0}$,
namely, $x_0=\mathbf{0}$. This implies that $\supp(D_{A}^{k_0}a)\subset B_0$.
Then, for any $\alpha\in\zz_+^n$ with $|\az|\leq s$ and $x\in\rn$, we have
$$\partial^\az\lf(\mathfrak{F}D_{A}^{k_0}a\r)(x)=\int_{B_0}
(-2\pi\imath\xi)^\az\lf(D_{A}^{k_0}a\r)(\xi)e^{-2\pi\imath \xi\cdot x}\,d\xi.$$
Let $P$ be the degree $s-|\az|$ Taylor expansion polynomial of the function
$\xi\mapsto e^{-2\pi\imath \xi\cdot x}$ centered at the origin $\mathbf{0}$.
By the vanishing moment of $a$, the Taylor remainder theorem and the H\"{o}lder
inequality, we find that
\begin{align}\label{3e5}
\lf|\partial^\az\lf(\mathfrak{F}D_{A}^{k_0}a\r)(x)\r|
&=\lf|\int_{B_0}(-2\pi\imath\xi)^\az
\lf(D_{A}^{k_0}a\r)(\xi)e^{-2\pi\imath \xi\cdot x}\,d\xi\r|\\
&=\lf|\int_{B_0}(-2\pi\imath\xi)^\az
\lf(D_{A}^{k_0}a\r)(\xi)\lf[e^{-2\pi\imath \xi\cdot x}-P(\xi)\r]\,d\xi\r|\noz\\
&\ls\int_{B_0}|\xi^\az|
\lf|a\lf(A^{k_0}\xi\r)\r||x|^{s-|\az|+1}|\xi|^{s-|\az|+1}\,d\xi\noz\\
&\ls|x|^{s-|\az|+1}\int_{B_0}
\lf|a\lf(A^{k_0}\xi\r)\r||\xi|^{s+1}\,d\xi
\ls|x|^{s-|\az|+1}b^{-k_0}\int_{B_{k_0}}|a(\xi)|\,d\xi\noz\\
&\ls|x|^{s-|\az|+1}b^{-k_0/r}\|a\|_{L^r(\rn)}.\noz
\end{align}
To obtain the other estimate, using the H\"{o}lder inequality directly,
we have
\begin{align*}
\lf|\partial^\az\lf(\mathfrak{F}D_{A}^{k_0}a\r)(x)\r|
&=\lf|\int_{B_0}(-2\pi\imath\xi)^\az
\lf(D_{A}^{k_0}a\r)(\xi)e^{-2\pi\imath \xi\cdot x}\,d\xi\r|\\
&\ls\int_{B_0}
|\xi|^{|\az|}\lf|a\lf(A^{k_0}\xi\r)\r|\,d\xi
\ls b^{-k_0}\int_{B_{k_0}}|a(\xi)|\,d\xi\\
&\ls b^{-k_0/r}\|a\|_{L^r(\rn)},
\end{align*}
which, combined with \eqref{3e5}, completes the proof of \eqref{3e1}
and hence of Lemma \ref{3l1}.
\end{proof}

From Lemma \ref{3l1}, we deduce a uniform estimate on the Fourier
transform of atoms, which is later used to prove Theorem \ref{3t1}.

\begin{lemma}\label{3l2}
Let $p(\cdot)\in C^{\log}(\rn)$ with $p_+\in(0,1]$, $r$ and $s$ be as
in Definition \ref{3d3}(ii). Then there exists a positive constant $C$
such that, for any $(p(\cdot),r,s)$-atom and $x\in\rn$,
\begin{align}\label{3e3}
\lf|\widehat{a}(x)\r|\leq C\max\lf\{\lf[\rho_*(x)\r]^{\frac1{p_-}-1},\,
\lf[\rho_*(x)\r]^{\frac1{p_+}-1}\r\},
\end{align}
where $\rho_*$ is the homogeneous quasi-norm with respect to $A^*$ and $p_-$, $p_+$
are as in \eqref{2e3}.
\end{lemma}

To show Lemma \ref{3l2}, we need the following inequalities,
which are just \cite[p.\,11, Lemma 3.2]{mb03}.

\begin{lemma}\label{3l3}
Let $A$ be some fixed dilation. Then there exists a positive constant $C$,
depending only on $A$, such that, for any $x\in\rn$,
\begin{equation*}
\frac{1}{C} [\rho(x)]^{\ln \lambda_-/\ln b} \leq \left|x\right|
\leq C[\rho(x)]^{\ln \lambda_+/\ln b}\qquad {\rm when}\ \rho(x)\in(1,\fz),
\end{equation*}
and
\begin{equation*}
\frac{1}{C} [\rho(x)]^{\ln \lambda_+/\ln b} \leq \left|x\right|
\leq C[\rho(x)]^{\ln \lambda_-/\ln b}\qquad {\rm when}\ \rho(x)\in[0,1].
\end{equation*}
\end{lemma}

We now give the proof of Lemma \ref{3l2}.

\begin{proof}[Proof of Lemma \ref{3l2}]
Let $a$ be a $(p(\cdot),r,s)$-atom supported in $x_0+B_{k_0}$ with some $x_0\in\rn$
and $k_0\in\zz$. Then, from \eqref{3e1} with $\az=(\overbrace{0,\ldots,0}^{n\ \rm times})$,
it follows that, for any $x\in\rn$,
\begin{align*}
\lf|\widehat{a}(x)\r|
&=\lf|b^{k_0}\lf(D_{A^*}^{k_0}\mathfrak{F}D_{A}^{k_0}a\r)(x)\r|\\
&=\lf|b^{k_0}\lf(\mathfrak{F}D_{A}^{k_0}a\r)\lf((A^*)^{k_0}x\r)\r|\\
&\ls b^{k_0}b^{-k_0/r}\|a\|_{L^r(\rn)}\min\lf\{1,\lf|(A^*)^{k_0}x\r|^{s+1}\r\},
\end{align*}
which, together with the size condition of $a$, implies that
\begin{align}\label{3e4}
\lf|\widehat{a}(x)\r|
&\ls b^{k_0}\lf\|\mathbf{1}_{B_{k_0}}\r\|_{\lv}^{-1}
\min\lf\{1,\lf|(A^*)^{k_0}x\r|^{s+1}\r\}\\
&\ls b^{k_0}\max\lf\{b^{-\frac{k_0}{p_-}},\,b^{-\frac{k_0}{p_+}}\r\}
\min\lf\{1,\lf|(A^*)^{k_0}x\r|^{s+1}\r\}.\noz
\end{align}
To obtain \eqref{3e3}, we next consider two cases:
$\rho_*(x)\leq b^{-k_0}$ and $\rho_*(x)>b^{-k_0}$.

\emph{Case 1).} $\rho_*(x)\leq b^{-k_0}$. In this case, note that
$\rho_*((A^*)^{k_0}x)\leq 1$. By \eqref{3e4}, Lemma \ref{3l3} and the fact that
$$1-\frac1{p_+}+(s+1)\frac{\ln\lambda_-}{\ln b}
\geq1-\frac1{p_-}+(s+1)\frac{\ln\lambda_-}{\ln b}>0$$
[see \eqref{2e3} and \eqref{3e2}], we conclude that, for any $x\in\rn$
satisfying $\rho_*(x)\leq b^{-k_0}$,
\begin{align}\label{3e12}
\lf|\widehat{a}(x)\r|
&\ls b^{k_0}\max\lf\{b^{-\frac{k_0}{p_-}},\,b^{-\frac{k_0}{p_+}}\r\}
\lf[\rho_*\lf((A^*)^{k_0}x\r)\r]^{(s+1)\frac{\ln\lambda_-}{\ln b}}\\
&\sim\max\lf\{b^{k_0[1-\frac1{p_-}+(s+1)\frac{\ln\lambda_-}{\ln b}]},\,
b^{k_0[1-\frac1{p_+}+(s+1)\frac{\ln\lambda_-}{\ln b}]}\r\}
\lf[\rho_*(x)\r]^{(s+1)\frac{\ln\lambda_-}{\ln b}}\noz\\
&\ls\max\lf\{\lf[\rho_*(x)\r]^{\frac1{p_-}-1},\,
\lf[\rho_*(x)\r]^{\frac1{p_+}-1}\r\}.\noz
\end{align}
This proves \eqref{3e3} for Case 1).

\emph{Case 2).} $\rho_*(x)>b^{-k_0}$. In this case, note that $\rho_*((A^*)^{k_0}x)>1$.
By \eqref{3e4}, Lemma \ref{3l3} again and the fact that
$$\frac1{p_-}-1\geq\frac1{p_+}-1\geq0,$$
we know that, for any $x\in\rn$ satisfying $\rho_*(x)>b^{-k_0}$,
\begin{align*}
\lf|\widehat{a}(x)\r|
&\ls b^{k_0}\max\lf\{b^{-\frac{k_0}{p_-}},\,b^{-\frac{k_0}{p_+}}\r\}\\
&\sim \max\lf\{b^{(1-\frac1{p_-})k_0},\,b^{(1-\frac1{p_+})k_0}\r\}\\
&\ls\max\lf\{\lf[\rho_*(x)\r]^{\frac1{p_-}-1},\,
\lf[\rho_*(x)\r]^{\frac1{p_+}-1}\r\}.
\end{align*}
This finishes the proof of \eqref{3e3} for Case 2) and hence of Lemma \ref{3l2}.
\end{proof}

Via borrowing some ideas from the proof of \cite[Lemma 5.9]{zsy16},
we obtain the following technical lemma.

\begin{lemma}\label{3l5}
Let $p(\cdot)\in\cp(\rn)$. Then, for any $\{\lz_i\}_{i\in\nn}\subset\mathbb{C}$
and $\{B^{(i)}\}_{i\in\nn}\subset\mathfrak{B}$,
$$\sum_{i\in\nn}|\lz_i|\le \lf\|\lf\{\sum_{i\in\nn}
\lf[\frac{|\lz_i|\mathbf{1}_{B^{(i)}}}{\|\mathbf{1}_{B^{(i)}}\|_{\lv}}\r]^
{\underline{p}}\r\}^{1/{\underline{p}}}\r\|_{\lv},$$
where $\underline{p}$ is as in \eqref{2e3}.
\end{lemma}

\begin{proof}
Let $\lz:=\sum_{i\in\nn}|\lz_i|$. Note that,
for any $\{\lz_i\}_{i\in\nn}\subset\mathbb{C}$ and $t\in(0,1]$,
$$\lf(\sum_{i\in\nn}|\lz_i|\r)^{t}\le \sum_{i\in\nn}|\lz_i|^{t}.$$
By the fact that $\underline{p}\in(0,1]$ [see \eqref{2e3}], we find that
\begin{align*}
\lf\|\lf\{\sum_{i\in\nn}
\lf[\frac{|\lz_i|\mathbf{1}_{B^{(i)}}}{\lz\|\mathbf{1}_{B^{(i)}}\|_{\lv}}\r]^
{\underline{p}}\r\}^{1/{\underline{p}}}\r\|_{\lv}
&\ge\lf\|\sum_{i\in\nn}
\frac{|\lz_i|\mathbf{1}_{B^{(i)}}}{\lz\|\mathbf{1}_{B^{(i)}}\|_{\lv}}\r\|_{\lv}\\
&\ge\sum_{i\in\nn}\frac{|\lz_i|}{\lz}\lf\|
\frac{\mathbf{1}_{B^{(i)}}}{\|\mathbf{1}_{B^{(i)}}\|_{\lv}}\r\|_{\lv}=1,
\end{align*}
which implies the desired conclusion and hence completes the proof of Lemma \ref{3l5}.
\end{proof}

To prove Theorem \ref{3t1}, we also need the following atomic characterizations
of $\vh$ established in \cite[Theorem 4.8]{lwyy18}.

\begin{lemma}\label{3l4}
Let $p(\cdot)\in C^{\log}(\rn)$, $r\in(\max\{p_+,1\},\fz]$ with $p_+$
as in \eqref{2e3}, $s$ be as in \eqref{3e2} and
$N\in\mathbb{N}\cap[\lfloor(\frac1{\underline{p}}-1)\frac{\ln b}{\ln
\lambda_-}\rfloor+2,\fz)$ with $\underline{p}$ as in \eqref{2e3}.
Then $\vh=\vah$ with equivalent quasi-norms.
\end{lemma}

Now, we show Theorem \ref{3t1}.

\begin{proof}[Proof of Theorem \ref{3t1}]
Let $p(\cdot)\in C^{\log}(\rn)$, $r\in(\max\{p_+,1\},\fz]$, $s$ be as in \eqref{3e2}
and $f\in\vh$. Then, from Lemma \ref{3l4} and Definition \ref{3d3}(ii), we deduce that
there exist a sequence
$\{\lz_i\}_{i\in\nn}\subset\mathbb{C}$ and a sequence of $(p(\cdot),r,s)$-atoms,
$\{a_i\}_{i\in\nn}$, supported, respectively, in
$\{B^{(i)}\}_{i\in\nn}\subset\B$ such that
\begin{align*}
f=\sum_{i\in\nn}\lz_ia_i
\quad\mathrm{in}\quad\cs'(\rn),
\end{align*}
and
\begin{align}\label{3e7}
\|f\|_{\vh}\sim\lf\|\lf\{\sum_{i\in\nn}
\lf[\frac{|\lz_i|\mathbf{1}_{B^{(i)}}}{\|\mathbf{1}_{B^{(i)}}\|_{\lv}}\r]^
{\underline{p}}\r\}^{1/\underline{p}}\r\|_{\lv}.
\end{align}
Therefore, by the continuity of Fourier transform on $\cs'(\rn)$, we know that
\begin{align}\label{3e10}
\widehat{f}=\sum_{i\in\nn}\lz_i\widehat{a_i}
\quad\mathrm{in}\quad\cs'(\rn).
\end{align}
Moreover, for any $i\in\nn$, $a_i\in L^1(\rn)$ implies that
$\widehat{a_i}\in L^\fz(\rn)$. By this,
Lemmas \ref{3l2} and \ref{3l5}, and \eqref{3e7}, we conclude that, for any $x\in\rn$,
\begin{align}\label{3e8}
\sum_{i\in\nn}|\lz_i||\widehat{a_i}(x)|
&\ls\sum_{i\in\nn}|\lz_i|
\max\lf\{\lf[\rho_*(x)\r]^{\frac1{p_-}-1},\,\lf[\rho_*(x)\r]^{\frac1{p_+}-1}\r\}\\
&\ls\|f\|_{\vh}
\max\lf\{\lf[\rho_*(x)\r]^{\frac1{p_-}-1},\,\lf[\rho_*(x)\r]^{\frac1{p_+}-1}\r\}\noz\\
&<\fz.\noz
\end{align}
Thus, for any $x\in\rn$,
\begin{align}\label{3e9}
F(x):=\sum_{i\in\nn}\lz_i\widehat{a_i}(x)
\end{align}
makes sense pointwisely and
\begin{align*}
|F(x)|\ls\|f\|_{\vh}
\max\lf\{\lf[\rho_*(x)\r]^{\frac1{p_-}-1},\,\lf[\rho_*(x)\r]^{\frac1{p_+}-1}\r\}.
\end{align*}

We next prove that the above function $F$ is continuous on $\rn$. To do this,
we only need to show that $F$ is continuous on any compact subset of $\rn$.
Note that, for any given compact subset $X$, there exists a positive constant
$C_{(A,X)}$, depending only on the dilation $A$ and $X$, such that
$\rho_*(\cdot)\leq C_{(A,X)}$ holds true absolutely on $X$. From this and
the estimate of \eqref{3e8}, it follows that, for any $x\in X$,
\begin{align*}
\sum_{i\in\nn}|\lz_i||\widehat{a_i}(x)|
&\ls\max\lf\{\lf[C_{(A,X)}\r]^{\frac1{p_-}-1},\,\lf[C_{(A,X)}\r]^{\frac1{p_+}-1}\r\}
\|f\|_{\vh},
\end{align*}
Therefore, the summation $\sum_{i\in\nn}\lz_i\widehat{a_i}(\cdot)$ converges uniformly
on $X$. This, combined with the fact that, for any $i\in\nn$, $\widehat{a_i}$ is continuous,
implies that $F$ is also continuous on any compact subset $X$ and hence on $\rn$.

Finally, to complete the proof of Theorem \ref{3t1}, by \eqref{3e10} and \eqref{3e9},
it suffices to show that
\begin{align}\label{3e11}
F=\sum_{i\in\nn}\lz_i\widehat{a_i}
\quad\mathrm{in}\quad\cs'(\rn).
\end{align}
To this end, by Lemma \ref{3l2} and the definition of Schwartz functions, we find that,
for any $\phi\in\cs(\rn)$ and $i\in\nn$,
\begin{align*}
&\lf|\int_{\rn}\widehat{a_i}(x)\phi(x)\,dx\r|\\
&\quad\leq\sum_{k=1}^\fz\int_{(A^*)^{k+1}B^*_0\setminus(A^*)^{k}B^*_0}
\max\lf\{\lf[\rho_*(x)\r]^{\frac1{p_-}-1},\,\lf[\rho_*(x)\r]^{\frac1{p_+}-1}\r\}
|\phi(x)|\,dx+\|\phi\|_{L^1(\rn)}\\
&\quad\ls\sum_{k=1}^\fz b^kb^{k(\frac1{p_-}-1)}b^{-k(\lceil\frac1{p_-}-1\rceil+2)}
+\|\phi\|_{L^1(\rn)}\\
&\quad\sim1,
\end{align*}
where $B^*_0$ is the unit dilated ball with respect to $A^*$ and,
for any $t\in\mathbb{R}$, $\lceil t\rceil$ denotes the least integer not less than $t$.
By this, Lemma \ref{3l5} and \eqref{3e7}, we further have
\begin{align*}
\lim_{K\to\fz}\sum_{i=K+1}^\fz|\lz_i|\lf|\int_{\rn}\widehat{a_i}(x)\phi(x)\,dx\r|
\ls\lim_{K\to\fz}\sum_{i=K+1}^\fz|\lz_i|=0,
\end{align*}
which implies that, for any $\phi\in\cs(\rn)$,
$$\langle F,\phi\rangle
=\lim_{K\to\fz}\lf\langle\sum_{i=1}^K\lz_i\widehat{a_i},\,\phi\r\rangle.$$
This finishes the proof of \eqref{3e11} and hence of Theorem \ref{3t1}.
\end{proof}

\begin{remark}\label{3r1}
\begin{enumerate}
\item[(i)]
When $p(\cdot)\equiv p\in(0,1]$, the Hardy space $\vh$ in Theorem \ref{3t1} coincides with
the anisotropic Hardy space $H^p_A(\rn)$ from \cite{mb03}, and the inequality \eqref{3e6}
becomes
\begin{align*}
|F(x)|\le C\|f\|_{H^{p}_A(\rn)}[\rho_\ast(x)]^{\frac1{p}-1}
\end{align*}
with $C$ as in \eqref{3e6}. In this case, Theorem \ref{3t1} is just \cite[Theorem 1]{bw13}.
\item[(ii)]
Let $f\in \vh\cap L^1({\mathbb{R}^n})$. In this case, we have $F=\widehat{f}$ and,
using the inequality \eqref{3e6} with $x=\mathbf{0}$,
we have $\widehat{f}(\mathbf{0})=0$.
This implies that the function $f\in \vh\cap L^1({\mathbb{R}^n})$
has a vanishing moment, which elucidates the necessity of the vanishing
moment of atoms in some sense.
\item[(iii)]
Very recently, in \cite[Theorem 2.4]{hcy21}, Huang et al. obtained a result similar
to Theorem \ref{3t1} with the Hardy space $\vh$ replaced by the anisotropic mixed-norm
Hardy space ${H_{\vec{a}}^{\vec{p}}(\rn)}$, where
$$\vec{a}:=(a_1,\ldots,a_n)\in [1,\fz)^n\quad
{\rm and}\quad \vec{p}:=(p_1,\ldots, p_n)\in (0,1]^n.$$
We should point out that the integrable exponent of the anisotropic mixed-norm Hardy space ${H_{\vec{a}}^{\vec{p}}(\rn)}$ is a vector $\vec{p}\in(0,1]^n$, whose associated basic
function space is the mixed-norm Lebesgue space ${L}^{\vec{p}}(\rn)$
which has different orders of integrability in different variables;
however, the integrable exponent of
the Hardy space $H^{p(\cdot)}_A(\rn)$ investigated in the present article is
a variable exponent function,
$$p(\cdot):\ \rn\to(0,1],$$
whose associated basic function space is the variable Lebesgue space $L^{p(\cdot)}(\rn)$;
Obviously, ${L}^{\vec{p}}(\rn)$ and $L^{p(\cdot)}(\rn)$ cannot cover each other,
so do \cite[Theorem 2.4]{hcy21} and Theorem \ref{3t1} of the present article.
\end{enumerate}
\end{remark}

\section{Applications\label{s4}}

As applications of Theorem \ref{3t1}, in this section, we first present a higher order
convergence of the function $F$ given in Theorem \ref{3t1} at the point $\mathbf{0}$.
Then we obtain an analogue of the Hardy--Littlewood inequality in the variable anisotropic
Hardy space setting.

We begin with the following result.

\begin{theorem}\label{4t1}
Let $p(\cdot)\in C^{\log}(\rn)$ satisfy $0<p_-\le p_+\le1$,
where $p_-,\ p_+$ are as in \eqref{2e3}. Then, for any
$f\in\vh$, there exists a continuous function $F$ on $\rn$ such that
$\widehat{f}=F$ in $\cs'(\rn)$
and
\begin{align}\label{4e1}
\lim_{|x|\to0^+}\frac{F(x)}{[\rho_*(x)]^{\frac1{p_-}{-1}}}=0,
\end{align}
where $\rho_*$ is the homogeneous quasi-norm with respect to $A^*$.
\end{theorem}

\begin{proof}
Let $p(\cdot)\in C^{\log}(\rn)$, $r\in(\max\{p_+,1\},\fz]$, $s$ be as in \eqref{3e2}
and $f\in\vh$. Then, by Lemma \ref{3l4} and Definition \ref{3d3}(ii), we know that
there exist a sequence
$\{\lz_i\}_{i\in\nn}\subset\mathbb{C}$ and a sequence of $(p(\cdot),r,s)$-atoms,
$\{a_i\}_{i\in\nn}$, supported, respectively, in
$\{B^{(i)}\}_{i\in\nn}\subset\B$ such that
\begin{align*}
f=\sum_{i\in\nn}\lz_ia_i
\quad\mathrm{in}\quad\cs'(\rn),
\end{align*}
and
\begin{align}\label{4e5}
\|f\|_{\vh}\sim\lf\|\lf\{\sum_{i\in\nn}
\lf[\frac{|\lz_i|\mathbf{1}_{B^{(i)}}}{\|\mathbf{1}_{B^{(i)}}\|_{\lv}}\r]^
{\underline{p}}\r\}^{1/\underline{p}}\r\|_{\lv}.
\end{align}
Moreover, from Theorem \ref{3t1} and its proof, we deduce that there exists a continuous
function on $\rn$, namely,
\begin{align}\label{4e2}
F=\sum_{i\in\nn}\lz_i\widehat{a_i}
\end{align}
such that $\widehat{f}=F$ in $\cs'(\rn)$.

Therefore, to complete the proof of Theorem \ref{4t1}, it suffices to show that \eqref{4e1}
holds true for the function $F$ as in \eqref{4e2}. Indeed, for any $(p(\cdot),r,s)$-atom
$a$ supported in a dilated ball $x_0+B_{k_0}$ with some $x_0\in\rn$ and $k_0\in\zz$, by
an argument similar to that used in Case 1) of the proof of Lemma \ref{3l2}, we conclude
that, for any $x\in\rn$ with $\rho_*(x)\leq b^{-k_0}$,
\begin{align*}
\lf|\widehat{a}(x)\r|
\ls\max\lf\{b^{k_0[1-\frac1{p_-}+(s+1)\frac{\ln\lambda_-}{\ln b}]},\,
b^{k_0[1-\frac1{p_+}+(s+1)\frac{\ln\lambda_-}{\ln b}]}\r\}
\lf[\rho_*(x)\r]^{(s+1)\frac{\ln\lambda_-}{\ln b}}.
\end{align*}
This, together with the fact that
$$1-\frac1{p_-}+(s+1)\frac{\ln\lambda_-}{\ln b}>0,$$
further implies that
\begin{align}\label{4e3}
\lim_{|x|\to0^+}\frac{|\widehat{a}(x)|}{[\rho_*(x)]^{\frac1{p_-}{-1}}}=0.
\end{align}
On another hand, from \eqref{4e2}, it follows that, for any $x\in\rn$,
\begin{align}\label{4e4}
\frac{|F(x)|}{[\rho_*(x)]^{\frac1{p_-}{-1}}}
\leq\sum_{i\in\nn}|\lz_i|\frac{|\widehat{a_i}(x)|}{[\rho_*(x)]^{\frac1{p_-}{-1}}}.
\end{align}
In addition, by Lemma \ref{3l5} and \eqref{4e5}, we find that
$\sum_{i\in\nn}|\lz_i|<\fz$. Thus, the equality \eqref{4e3} implies that, for any given
$\epsilon\in(0,1)$, there exists a positive constant $\nu$ such that, for any $i\in\nn$
and $x\in\rn$ with $|x|\leq\nu$,
$$\frac{|\widehat{a_i}(x)|}{[\rho_*(x)]^{\frac1{p_-}{-1}}}
<\frac\epsilon{\sum_{i\in\nn}|\lz_i|+1}.$$
By this and \eqref{4e4}, we know that, for any $x\in\rn$ with $|x|\leq\nu$,
$$\frac{|F(x)|}{[\rho_*(x)]^{\frac1{p_-}{-1}}}<\epsilon.$$
Thus,
\begin{align*}
\lim_{|x|\to0^+}\frac{F(x)}{[\rho_*(x)]^{\frac1{p_-}{-1}}}=0,
\end{align*}
which completes the proof of \eqref{4e1} and hence of Theorem \ref{4t1}.
\end{proof}

\begin{remark}
\begin{enumerate}
\item[(i)]
Similar to Remark \ref{3r1},
if $p(\cdot)\equiv p\in(0,1]$, then the Hardy space $\vh$ in Theorem \ref{4t1} coincides with
the anisotropic Hardy space $H^p_A(\rn)$ from \cite{mb03}. In this case, Theorem \ref{4t1} is just \cite[Corollary 6]{bw13}.
\item[(ii)]
By Theorem \ref{4t1} and Lemma \ref{3l3}, we find that
\begin{align}\label{1e6}
\lim_{|x|\to 0^+}\frac{F(x)}
{|x|^{\frac{\ln b}{\ln\lambda_+}(\frac 1{p_-}-1)}}=0.
\end{align}
Note that, when $p(\cdot)\equiv p\in(0,1]$ and
$A=d\,{\rm I}_{n\times n}$ for some $d\in\rr$ with $|d|\in(1,\fz)$,
here and thereafter, ${\rm I}_{n\times n}$ denotes the $n\times n$ \emph{unit matrix},
the Hardy space $\vh$ coincides with the classical Hardy space $H^p({\mathbb{R}^n})$ of
Fefferman and Stein \cite{fs72}. In this case, $\frac{\ln b}{\ln\lambda_+}=n$ and $p_-=p$,
and hence \eqref{1e6} goes back to the well-known result on $H^p({\mathbb{R}^n})$
(see \cite[p.\,128]{s93}).
\end{enumerate}
\end{remark}

As another application of Theorem \ref{3t1}, we also establish a variant of the
Hardy--Littlewood inequality in the variable anisotropic Hardy space setting
as follows.

\begin{theorem}\label{4t2}
Let $p(\cdot)\in C^{\log}(\rn)$ satisfy $0<p_-\le p_+\le1$,
where $p_-,\ p_+$ are as in \eqref{2e3}. Then, for any
$f\in\vh$, there exists a continuous function $F$ on $\rn$ such that
$\widehat{f}=F$ in $\cs'(\rn)$
and
\begin{align}\label{4e6}
\lf(\int_\rn|F(x)|^{p_+}\min\lf\{\lf[\rho_*(x)\r]^{p_+-\frac{p_+}{p_-}-1},\,
\lf[\rho_*(x)\r]^{p_+-2}\r\}\,dx\r)^{\frac1{p_+}}
\leq C\|f\|_{\vh},
\end{align}
where $\rho_*$ denotes the homogeneous quasi-norm with respect to $A^*$ and
$C$ is a positive constant depending only on $A$, $p_-$ and $p_+$.
\end{theorem}

\begin{proof}
Let $p(\cdot)\in C^{\log}(\rn)$ with $p_+\in(0,1]$, $s$ be as in \eqref{3e2}
and $f\in\vh$. Then, from Lemma \ref{3l4} and Definition \ref{3d3}(ii), we deduce that
there exist a sequence
$\{\lz_i\}_{i\in\nn}\subset\mathbb{C}$ and a sequence of $(p(\cdot),2,s)$-atoms,
$\{a_i\}_{i\in\nn}$, supported, respectively, in
$\{B^{(i)}\}_{i\in\nn}\subset\B$ such that
\begin{align*}
f=\sum_{i\in\nn}\lz_ia_i
\quad\mathrm{in}\quad\cs'(\rn),
\end{align*}
and
\begin{align}\label{4e7}
\lf\|\lf\{\sum_{i\in\nn}
\lf[\frac{|\lz_i|\mathbf{1}_{B^{(i)}}}{\|\mathbf{1}_{B^{(i)}}\|_{\lv}}\r]^
{\underline{p}}\r\}^{1/\underline{p}}\r\|_{\lv}\leq2\|f\|_{\vh}<\fz.
\end{align}
Moreover, by Theorem \ref{3t1} and its proof, we conclude that there exists a continuous
function on $\rn$, namely,
\begin{align}\label{4e8}
F=\sum_{i\in\nn}\lz_i\widehat{a_i}
\end{align}
such that $\widehat{f}=F$ in $\cs'(\rn)$.

Therefore, to prove Theorem \ref{4t2}, it suffices to show that \eqref{4e6}
holds true for the function $F$ as in \eqref{4e8}. To this end, by the
fact that $\underline{p}\le p_+\le1$ and the well-known inequality that,
for any
$\{\alpha_i\}_{i\in{\mathbb N}}\subset\mathbb{C}$
and $t\in(0,1]$,
\begin{align}\label{4e9}
\left[\sum_{i\in{\mathbb N}}|\alpha_i|\right]^{t}
\le \sum_{i\in{\mathbb N}}|\alpha_i|^{t}
\end{align}
as well as \eqref{4e7}, we find that
\begin{align}\label{4e10}
\left(\sum_{i\in{\mathbb N}}|\lambda_i|^{p_+}\right)^{1/p_+}
&=\left(\sum_{i\in{\mathbb N}}\left\|\frac{|\lambda_i|{\mathbf 1}_{B^{(i)}}}
{\|{\mathbf 1}_{B^{(i)}}\|_{L^{p(\cdot)}(\mathbb{R}^n)}}
\right\|_{\lv}^{p_+}\right)^{1/p_+}\\
&=\left(\sum_{i\in{\mathbb N}}
\left\|\frac{|\lambda_i|^{p_+}{\mathbf 1}_{B^{(i)}}}{\|{\mathbf 1}_{B^{(i)}}\|_{\lv}^{p_+}}
\right\|_{L^{p(\cdot)/p_+}({\mathbb{R}^n})}\right)^{1/p_+}\nonumber\\
&\le \left\|\sum_{i\in{\mathbb N}}
\left[\frac{|\lambda_i|{\mathbf 1}_{B^{(i)}}}{\|{\mathbf 1}_{B^{(i)}}\|_{\lv}}
\right]^{p_+}\right\|_{L^{p(\cdot)/p_+}({\mathbb{R}^n})}^{1/p_+}\nonumber\\
&=\lf\|\lf\{\sum_{i\in{\mathbb N}}\lf[\frac{|\lambda_i|{\mathbf 1}_{B^{(i)}}}
{\|{\mathbf 1}_{B^{(i)}}\|_{\lv}}
\right]^{p_+}\right\}^{1/p_+}\right\|_{L^{p(\cdot)}(\mathbb{R}^n)}\nonumber\\
&\le\left\|\left\{\sum_{i\in{\mathbb N}}
\left[\frac{|\lambda_i|{\mathbf 1}_{B^{(i)}}}{\|{\mathbf 1}_{B^{(i)}}\|_
{\lv}}\right]^{\underline{p}}\right\}^{1/\underline{p}}
\right\|_{\lv}\nonumber\\
&\le 2\|f\|_{\vh}.\nonumber
\end{align}

On another hand, by \eqref{4e8}, the fact that $p_+\in(0,1]$, \eqref{4e9}
and the Fatou lemma, it is easy to see that
\begin{align}\label{4e11}
&\int_{{\mathbb{R}^n}}|F(x)|^{p_+}
\min\left\{\left[\rho_{*}(x)\right]^{p_+-\frac {p_+}{p_-}-1},\,
\left[\rho_{*}(x)\right]^{p_+-2}\right\}\, dx\\
&\quad\le \sum_{i\in{\mathbb N}}|\lambda_i|^{p_+}\int_{{\mathbb{R}^n}}\left[\lf|\widehat{a_i}(x)\r|
\min\left\{\left[\rho_{*}(x)\right]^{1-\frac 1{p_-}-\frac 1{p_+}},\,
\left[\rho_{*}(x)\right]^{1-\frac 2{p_+}}\right\}\right]^{p_+}\,dx.\noz
\end{align}
If we can prove the following assertion: there exists a positive
constant $R$ such that, for any $(p(\cdot),2,s)$-atom $a$,
\begin{align}\label{4e12}
\left(\int_{{\mathbb{R}^n}}\left[\lf|\widehat{a}(x)\r|
\min\left\{\left[\rho_{*}(x)\right]^{1-\frac 1{p_-}-\frac 1{p_+}},\,
\left[\rho_{*}(x)\right]^{1-\frac 2{p_+}}
\right\}\right]^{p_+}\,dx\right)^{1/p_+}\le R,
\end{align}
then, by this assertion, \eqref{4e11} and \eqref{4e10}, we have
\begin{align*}
&\lf(\int_{{\mathbb{R}^n}}|F(x)|^{p_+}
\min\left\{\left[\rho_{*}(x)\right]^{p_+-\frac {p_+}{p_-}-1},\,
\left[\rho_{*}(x)\right]^{p_+-2}\right\}\, dx\r)^{1/{p_+}}\\
&\quad\le R\left(\sum_{i\in{\mathbb N}}|\lambda_i|^{p_+}\right)^{1/{p_+}}
\ls\|f\|_{\vh}.
\end{align*}
This is the desired conclusion \eqref{4e6}.

Thus, to complete the whole proof, it remains to show the assertion
\eqref{4e12}. Indeed, for any $(p(\cdot),2,s)$-atom $a$ supported in
a dilated ball $x_0+B_{k_0}$ with some $x_0\in\rn$ and $k_0\in\zz$,
we easily know that
\begin{align}\label{4e13}
&\left(\int_{{\mathbb{R}^n}}\left[\lf|\widehat{a}(x)\r|
\min\left\{\left[\rho_{*}(x)\right]^{1-\frac 1{p_-}-\frac 1{p_+}},\,
\left[\rho_{*}(x)\right]^{1-\frac 2{p_+}}\right\}\right]^{p_+}\,dx\right)^{1/{p_+}}\\
&\quad\ls\left(\int_{(A^*)^{-k_0+1}B_0^*}\left[\lf|\widehat{a}(x)\r|
\min\left\{\left[\rho_{*}(x)\right]^{1-\frac 1{p_-}-\frac 1{p_+}},\,
\left[\rho_{*}(x)\right]^{1-\frac 2{p_+}}\right\}\right]^{p_+}\,dx\right)^{1/{p_+}}\nonumber\\
&\qquad
+\left(\int_{((A^*)^{-k_0+1}B_0^*)^{\complement}}\left[\lf|\widehat{a}(x)\r|
\min\left\{\left[\rho_{*}(x)\right]^{1-\frac 1{p_-}-\frac 1{p_+}},\,
\left[\rho_{*}(x)\right]^{1-\frac 2{p_+}}\right\}\right]^{p_+}\,dx\right)^{1/{p_+}}\nonumber\\
&\quad=:{\rm I}_1+{\rm I}_2,\nonumber
\end{align}
where $B^*_0$ is the unit dilated ball with respect to $A^*$.

Let $\varepsilon$ be a fixed positive constant such that
$$1-\frac1{p_+}+(s+1)\frac{\ln\lambda_-}{\ln b}-\varepsilon
\geq1-\frac1{p_-}+(s+1)\frac{\ln\lambda_-}{\ln b}-\varepsilon>0.$$
Then, for ${\rm I}_1$, from the estimate of \eqref{3e12},
it follows that
\begin{align}\label{4e14}
{\rm I}_1
&\ls b^{k_0[1+(s+1)\frac{\ln\lambda_-}{\ln b}]}
\max\lf\{b^{\frac{k_0}{p_-}},\,b^{\frac{k_0}{p_+}}\r\}\\
&\hs\times\left(\int_{(A^*)^{-k_0+1}B_0^*}\left[
\min\left\{\left[\rho_{*}(x)\right]^{1-\frac 1{p_-}-\frac 1{p_+}+(s+1)\frac{\ln\lambda_-}{\ln b}},\,
\left[\rho_{*}(x)\right]^{1-\frac 2{p_+}+(s+1)\frac{\ln\lambda_-}{\ln b}}\right\}\right]^{p_+}\,dx\right)^{1/{p_+}}\noz\\
&\ls b^{k_0[1+(s+1)\frac{\ln\lambda_-}{\ln b}]}
\max\lf\{b^{\frac{k_0}{p_-}},\,b^{\frac{k_0}{p_+}}\r\}\noz\\
&\hs\times\min\lf\{b^{-k_0[1-\frac1{p_-}+(s+1)\frac{\ln\lambda_-}{\ln b}-\varepsilon]},\,b^{-k_0[1-\frac1{p_+}+(s+1)\frac{\ln\lambda_-}{\ln b}-\varepsilon]}\r\}\left(\int_{(A^*)^{-k_0+1}B_0^*}
\left[\rho_{*}(x)\right]^{\varepsilon p_+-1}\,dx\right)^{1/{p_+}}\noz\\
&\sim b^{k_0\varepsilon}\left[\sum_{k=-\fz}^0b^{-k_0+k}b^{(-k_0+k)(\varepsilon p_+-1)}\r]^{1/{p_+}}\sim1.\noz
\end{align}
To deal with ${\rm I}_2$, by the H\"{o}lder inequality, the Plancherel theorem,
the fact that $0<p_-\le p_+\le1$ and the size condition of $a$,
we conclude that
\begin{align*}
{\rm I}_2
&\ls\left\{\int_{((A^*)^{-k_0+1}B_0^*)^{\complement}}
\lf|\widehat{a}(x)\r|^2\,dx\right\}^{\frac12}\\
&\hs\hs\times\left\{\int_{((A^*)^{-k_0+1}B_0^*)^{\complement}}
\left[\min\left\{\left[\rho_{*}(x)\right]^{1-\frac 1{p_-}-\frac 1{p_+}},\,
\left[\rho_{*}(x)\right]^{1-\frac 2{p_+}}\right\}\right]^{\frac{2p_+}{2-p_+}}\,dx\right\}^{\frac{2-p_+}{2p_+}}\\
&\ls\|a\|_{L^2({\mathbb{R}^n})}\left\{\sum_{k=0}^\fz b^{-k_0+k}
\left[\min\left\{b^{(-k_0+k)(1-\frac 1{p_-}-\frac 1{p_+})},\,
b^{(-k_0+k)(1-\frac 2{p_+})}\right\}\right]^{\frac{2p_+}{2-p_+}}\right\}^{\frac{2-p_+}{2p_+}}\\
&\ls\|a\|_{L^2({\mathbb{R}^n})}\left\{b^{-k_0}
\left[\min\left\{b^{-k_0(1-\frac 1{p_-}-\frac 1{p_+})},\,
b^{-k_0(1-\frac 2{p_+})}\right\}\right]^{\frac{2p_+}{2-p_+}}\right\}^{\frac{2-p_+}{2p_+}}\\
&\ls\max\left\{b^{k_0(\frac12-\frac 1{p_-})},\,
b^{k_0(\frac12-\frac 1{p_+})}\right\}
\min\left\{b^{-k_0(\frac12-\frac 1{p_-})},\,
b^{-k_0(\frac12-\frac 1{p_+})}\right\}\\
&\sim 1.
\end{align*}
This, combined with \eqref{4e13} and \eqref{4e14}, implies that
\eqref{4e12} holds true and hence finishes the proof of Theorem \ref{4t2}.
\end{proof}

\begin{remark}
Recall that the well-known Hardy--Littlewood inequality for the
classical Hardy space $H^p({\mathbb{R}^n})$ is as follows:
Let $p\in(0,1]$. Then for each $f\in H^p({\mathbb{R}^n})$, we can find a continuous
function $F$ on ${\mathbb{R}^n}$ satisfying that $\widehat{f}=F$ in $\cs'({\mathbb{R}^n})$
and
\begin{align}\label{1e7}
\left[\int_{{\mathbb{R}^n}}\left|x\right|^{n(p-2)}
\left|F(x)\right|^p\,dx\right]^{1/p}\le C\left\|f\right\|_{H^p({\mathbb{R}^n})},
\end{align}
where $C$ is a positive constant independent of $f$
and $F$ (see \cite[p.\,128]{s93}).

We point out that the inequality \eqref{4e6} in Theorem \ref{4t2} is an analogue
of the Hardy--Littlewood inequality in the present setting. Indeed,
similar to Remark \ref{3r1},
when $p(\cdot)\equiv p\in(0,1]$, the Hardy space $\vh$ in Theorem \ref{4t2} becomes
the anisotropic Hardy space $H^p_A(\rn)$ from \cite{mb03}. In this case,
$p_+=p_-=p$ and hence Theorem \ref{4t2} is just \cite[Corollary 8]{bw13}.
Moreover, if
$A=d\,{\rm I}_{n\times n}$ for some $d\in\rr$ with $|d|\in(1,\fz)$,
then the anisotropic Hardy space $H^p_A(\rn)$ (namely, the Hardy space $\vh$ with
$p(\cdot)\equiv p\in(0,1]$) coincides with the classical Hardy space $H^p({\mathbb{R}^n})$ of
Fefferman and Stein \cite{fs72}. In this case, $\rho_*(x)\sim|x|^n$ for any $x\in\rn$,
and hence the Hardy--Littlewood inequality \eqref{4e6} is just \eqref{1e7}.
\end{remark}

\bigskip

\noindent Jun Liu

\medskip

\noindent  School of Mathematics,
China University of Mining and Technology,
Xuzhou 221116, Jiangsu, People's Republic of China

\smallskip

\noindent{\it E-mail:}
\texttt{junliu@cumt.edu.cn}

\end{document}